\documentclass{amsart}

\usepackage[latin1]{inputenc}
\usepackage[english]{babel}
\usepackage{indentfirst}
\usepackage{amssymb}
\usepackage{amsthm}
\usepackage{xcolor}
\usepackage[all]{xy}
\usepackage[mathscr]{eucal}

\newcommand{\impli}{\Rightarrow}
\newcommand{\Nat}{\mathbb{N}}

\newcommand{\erre}{\mathbb{R}}
\newcommand{\sub}{\subseteq}
\def\epsilon{\varepsilon}

\newtheorem{theorem}{Theorem}[section]

\newtheorem{corollary}[theorem]{Corollary}
\newtheorem{lemma}[theorem]{Lemma}

\theoremstyle{definition}
\newtheorem{definition}[theorem]{Definition}
\newtheorem{example}[theorem]{Example}
\newtheorem{remark}[theorem]{Remark}

\numberwithin{equation}{section}

\title{On $p$-Dunford integrable functions with values in Banach spaces}

\author{J.M. Calabuig}
\address{Instituto Universitario de Matem\'{a}tica Pura y Aplicada\\Universitat Polit\`{e}cnica de Val\`{e}ncia\\
Camino de Vera s/n\\ 46022 Valencia\\ Spain} \email{jmcalabu@mat.upv.es}

\author{J. Rodr\'{i}guez}
\address{Dpto. de Ingenier\'{i}a y Tecnolog\'{i}a de Computadores\\Facultad de Inform\'{a}tica\\
Universidad de Murcia\\ 30100 Espinardo (Murcia)\\ Spain} \email{joserr@um.es}

\author{P. Rueda}
\address{Dpto. de An\'{a}lisis Matem\'{a}tico \\ Facultad de Matem\'{a}ticas \\ Universidad de Val{e}ncia\\
Avda. Doctor Moliner 50\\ 46100 Burjassot (Valencia)\\ Spain} \email{pilar.rueda@uv.es}

\author{E.A. S\'{a}nchez-P\'{e}rez}
\address{Instituto Universitario de Matem\'{a}tica Pura y Aplicada\\ Universitat Polit\`{e}cnica de Val\`{e}ncia\\
Camino de Vera s/n\\ 46022 Valencia\\ Spain} \email{easancpe@mat.upv.es}



\thanks{Research partially supported by {\em Ministerio de Econom\'{i}a y Competitividad} and {\em FEDER} 
under projects MTM2014-53009-P (J.M. Calabuig), MTM2014-54182-P (J. Rodr\'{i}guez) and MTM2016-77054-C2-1-P (P. Rueda and E.A. S\'{a}nchez-P\'{e}rez). 
The second author was also partially supported by project 19275/PI/14 funded by {\em Fundaci\'on S\'eneca -- Agencia de Ciencia y Tecnolog\'\i a 
de la Regi\'on de Murcia} within the framework of PCTIRM 2011-2014. }

\begin{document}

\begin{abstract}
Let  $(\Omega,\Sigma,\mu)$ be a complete probability space, $X$ a Banach space and $1\leq p<\infty$. 
In this paper we discuss several aspects of $p$-Dunford integrable functions $f:\Omega \to X$. 
Special attention is paid to the compactness of the Dunford operator of~$f$. We also study 
the $p$-Bochner integrability of the composition $u\circ f:\Omega \to Y$, where $u$ is a $p$-summing operator from~$X$ to another Banach space~$Y$. 
Finally, we also provide some tests of $p$-Dunford integrability by using $w^*$-thick subsets of~$X^*$.
\end{abstract}

\maketitle

\section{Introduction}

Throughout this paper $(\Omega, \Sigma, \mu)$ is a complete probability space, $X$ is a Banach space and $1\leq p<\infty$.
Dunford and Pettis integrable functions $f:\Omega \to X$ have been widely studied over the years and their properties
are nowadays well understood, see e.g. \cite{Musial,mus3,tal}. However, it seems that this is not the case
for $p$-Dunford and $p$-Pettis integrable functions when $p>1$. This paper aims to contribute to fill in this gap. The following 
definition goes back to Pettis~\cite{Pettis}. 

\begin{definition}\label{definition:pPettis}
A function $f:\Omega \to X$ is called: 
\begin{enumerate}
\item[(i)] {\it $p$-Dunford integrable} if $\langle f,x^*\rangle\in L^p(\mu)$ for every $x^*\in X^*$;
\item[(ii)] {\it $p$-Pettis integrable} if it is $p$-Dunford integrable and Pettis integrable.
\end{enumerate}
\end{definition}
 
As usual, for any $f:\Omega \to X$ and $x^*\in X^*$, the scalar function $\langle f,x^*\rangle:\Omega\to \mathbb R$ 
is defined by $\langle f,x^*\rangle(\omega):=\langle f(\omega),x^*\rangle$ for all $\omega\in \Omega$.

Some scattered results about $p$-Dunford and $p$-Pettis integrable functions can be found in \cite{dia-alt,dre-alt,fre-alt-1,fre-alt-2,fre-alt-3,Pettis}.
Most of them are restricted to the case of {\em strongly measurable} functions or, equivalently,
{\em separable} Banach spaces. For instance, if $p>1$, 
then every strongly measurable $p$-Dunford integrable function $f:\Omega \to X$ is $p$-Pettis integrable, see~\cite[Corollary~5.31]{Pettis}
(cf. \cite[Corollary~5.2]{Musial}). In this paper we deal with $p$-Dunford and $p$-Pettis
integrable functions which are not necessarily strongly measurable. 

Let us summarize the content of this work. 

In Section~\ref{section:preliminaries} we introduce the basic terminology and 
include some preliminaries on $p$-Dunford and $p$-Pettis integrable functions.

In Section~\ref{section:compact} we discuss the compactness
of the {\em Dunford operator} 
$$
	T_f^p: L^{p'}(\mu) \to X^{**}, 
$$ 
associated to a $p$-Dunford integrable function $f:\Omega \to X$. As usual,
$1<p'\leq \infty$ denotes the conjugate exponent of~$p$, i.e. $\frac 1p+\frac 1{p'}=1$. The operator $T_f^p$ is defined
as the adjoint of the operator 
\begin{equation}\label{eqn:Sf}
	S_f^p:X^* \to L^p(\mu), \quad
	S_f^p(x^*):=\langle f,x^*\rangle.
\end{equation}
Compactness of the Dunford operator is important for applications and  
some authors add it to the definition of $p$-Pettis integrable function, see~\cite{fre-alt-1}.
If $f$ is $p$-Pettis integrable, then $T_f^p$ takes values in~$X$. 
In this case, the compactness of $T_f^p$ is equivalent
to the approximation of~$f$ by simple functions in the $p$-Pettis norm (Theorem~\ref{theorem:compactVSapprox}). 
An example of Pettis~\cite{Pettis} (cf. Example~\ref{example:Pettis}) already showed that $T_f^p$ might not be compact for $p>1$
even for a strongly measurable $p$-Pettis integrable function~$f$. In the case $p=1$, the counterexamples for Pettis integrable functions 
involve necessarily non strongly measurable functions and are far from elementary, the first one being
constructed by Fremlin and Talagrand~\cite{fre-tal}. We show that for a $p$-Dunford integrable function $f$,
the operator $T_f^p$ is compact if and only if (i) $T_f^1: L^\infty(\mu) \to X^{**}$ is compact, and (ii) the family of real-valued functions
$$
	Z_f^p:=\big\{|\langle f,x^*\rangle|^p: \, x^*\in B_{X^*}\big\},
$$
is a uniformly integrable subset of~$L^1(\mu)$ (Theorem~\ref{theorem:CompactDunfordOperator}). 
Condition (i) follows automatically from~(ii) whenever $f$ is strongly measurable, but also in many other cases,
e.g. if $\mu$ is perfect or if $X\not\supseteq \ell^1(\aleph_1)$, where $\aleph_1$ denotes the first
uncountable cardinal (see Corollary~\ref{cor:DCP}).

In Section~\ref{section:summing} we study the integrability of the composition 
$u\circ f:\Omega \to Y$, where $u$ is a {\em $p$-summing} operator from~$X$ to another Banach space~$Y$
and $f:\Omega \to X$ is a $p$-Dunford integrable function. For $p=1$, Diestel~\cite{die2} proved that
$u\circ f$ is Bochner integrable whenever $f$ is strongly measurable and Pettis integrable. 
As remarked in \cite[p.~56]{die-alt}, his argument can be easily modified for arbitrary $1\leq p<\infty$ to obtain that 
$u\circ f$ is $p$-Bochner integrable whenever $f$ is strongly measurable and $p$-Dunford integrable.
Several papers discussed such type of questions for $p=1$ beyond the strongly measurable case,
see \cite{bel-dow,dev-rod,lew3,rod3,rod10}. An unpublished result of Lewis~\cite{lew3}, rediscovered independently in~\cite[Theorem~2.3]{rod3}, states that 
for $p=1$ the composition $u\circ f$ is at least {\em scalarly equivalent} to a Bochner integrable function if $f$ is Dunford integrable.
Here we generalize this result for the range $1\leq p<\infty$ (Theorem~\ref{theorem:pSumming-scalar}) and provide many
examples of Banach spaces $X$ for which $u\circ f$ is actually $p$-Bochner integrable 
if $f$ is $p$-Dunford integrable (Corollary~\ref{corollary:pSumming-strong}). To this end we need some auxiliary results
on the $p$-variation and $p$-semivariation of a vector measure which might be of independent interest.

Finally, in Section~\ref{section:thick} we give some criteria to check the $p$-Dunford integrability
of a function $f:\Omega \to X$ by looking at the family of real-valued functions
$$
	Z_{f,\Gamma}:=\{\langle f,x^*\rangle: \, x^*\in \Gamma\},
$$
for some set $\Gamma \sub X^*$. Fonf proved that $f$ is Dunford integrable
if $X$ is separable, $X\not\supseteq c_0$ and $\langle f,x^* \rangle\in L^1(\mu)$
for every extreme point $x^*$~of~$B_{X^*}$ (see~\cite[Theorem~4]{fon2-J}). This result is based on the striking fact
that if $X\not \supseteq c_0$, then the set of all extreme points of~$B_{X^*}$ 
(or, more generally, any James boundary of~$B_{X^*}$) is $w^*$-thick, see \cite[Theorem~1]{fon2-J}
(cf. \cite[Theorem~2.3]{nyg}).
A set $\Gamma \sub X^*$ is said to be {\em $w^*$-thick} if whenever $\Gamma=\bigcup_{n\in \Nat} \Gamma_n$ for some 
increasing sequence $(\Gamma_n)$ of sets, 
there is $n\in \Nat$ such that $\overline{{\rm aco}}^{w^*}(\Gamma_n)$ contains a ball centered at~$0$.
This concept is useful to check several properties without testing on the whole dual,
see \cite{nyg} for more information. In \cite{abr-alt} it was pointed out that
if $X$ is separable, then a function $f:\Omega \to X$ is Dunford integrable whenever $Z_{f,\Gamma} \sub L^1(\mu)$
for some $w^*$-thick set $\Gamma \sub X^*$. As an application of our main theorem
of this section (Theorem~\ref{theorem:pDunford-thick}), we extend the result of \cite{abr-alt} to the range $1\leq p<\infty$
and a wide class of Banach spaces, namely, those having Efremov's property~($\mathcal{E}$)
(Corollary~\ref{corollary:pDunford-thick}). This also complements similar results in~\cite{raj-rod} dealing
with scalarly bounded functions.

\section{Preliminaries}\label{section:preliminaries}

We follow standard Banach space terminology as it can be found in \cite{die-uhl-J} and \cite{fab-ultimo}.
All our Banach spaces are real. An {\em operator} between Banach spaces is a  continuous linear map. 
Given a Banach space~$Z$, its norm is denoted by $\|\cdot\|_Z$ or simply
$\|\cdot\|$ if no confussion arises. We write $B_Z=\{z\in Z:\|z\|\leq 1\}$ (the closed unit ball of~$Z$) and
$S_Z=\{z\in Z:\|z\|= 1\}$ (the unit sphere of~$Z$). The topological dual of~$Z$
is denoted by~$Z^*$. The weak topology on~$Z$ and the weak$^*$ topology on~$Z^*$ 
are denoted by $w$ and~$w^*$, respectively. The evaluation of $z^*\in Z^*$ at $z\in Z$ is denoted by either
$\langle z,z^*\rangle$ or $\langle z^*,z\rangle$. A {\em subspace} of~$Z$ is a closed linear subspace.
Given another Banach space~$Y$, we write $Z\not\supseteq Y$ if $Z$ contains no subspace isomorphic to~$Y$. 
The absolutely convex hull of a set $S \sub Z$ is denoted by ${\rm aco}(S)$.

The characteristic function of $A\in \Sigma$ is denoted by $\chi_A$.
A set $H\sub L^1(\mu)$ is called {\em uniformly integrable} if it is bounded
and for every $\epsilon>0$ there is $\delta>0$ such that $\sup_{h\in H}\int_A |h| \, d\mu\leq \epsilon$
for every $A\in \Sigma$ with $\mu(A)\leq\delta$. This is equivalent to saying that $H$ is relatively weakly compact in~$L^1(\mu)$,
see e.g. \cite[p.~76, Theorem~15]{die-uhl-J}. By using H{\"o}lder's inequality it is easy to check, for $p>1$, that
any bounded subset of $L^p(\mu)$ is uniformly integrable as a subset of~$L^1(\mu)$.

A function $f:\Omega \to X$ is called:
\begin{itemize}
\item {\em simple} if it can be written as $f=\sum_{i=1}^nx_i\chi_{A_i}$, where $n\in \Nat$, $x_i\in X$ and $A_i\in \Sigma$
for every $i=1,\dots,n$;
\item {\em scalarly bounded} if there is a constant $M>0$ such that, for each $x^*\in X^*$, we have
$|\langle f,x^*\rangle|\leq M\|x^*\|$ $\mu\mbox{-a.e.}$ (the exceptional set depending on~$x^*$); 
\item {\em scalarly measurable} if $\langle f,x^*\rangle$ is measurable for every $x^*\in X^*$;
\item {\it strongly measurable} if there is a sequence of simple functions $f_n:\Omega\to X$ such that
$f_n(\omega)\to f(\omega)$ in norm for $\mu$-a.e. $\omega \in \Omega$.
\end{itemize}

The celebrated Pettis' measurability theorem (see e.g. \cite[p.~42, Theorem~2]{die-uhl-J}) states that $f$ is strongly measurable if and only
if it is scalarly measurable and there is $A\in \Sigma$ with $\mu(\Omega \setminus A)=0$ such that $f(A)$ is separable.

Two functions $f,g:\Omega \to X$ are said to be {\em scalarly equivalent}
if for each $x^*\in X^*$ we have $\langle f,x^*\rangle=\langle g,x^*\rangle$ $\mu$-a.e.
(the exceptional set depending on~$x^*$).

Given any Dunford (i.e. $1$-Dunford) integrable function $f:\Omega \to X$, there is a finitely additive measure $\nu_f:\Sigma \to X^{**}$ satisfying
$$
	\langle \nu_f(A),x^*\rangle=\int_A \langle f,x^*\rangle \, d\mu\quad
	\mbox{for all }A\in \Sigma \mbox{ and } x^*\in X^*. 
$$
As usual, we also write $\int_A f \, d\mu:=\nu_f(A)$. Recall that $f$ is said to be {\em Pettis integrable} if $\int_A f \, d\mu\in X$
for all $A\in \Sigma$. 

\begin{remark}\label{remark:basic-pDunford}
Let $f:\Omega \to X$ be a {\em $p$-Dunford integrable} function. Then:
\begin{enumerate}
\item[(i)] $f$ is Dunford integrable.
\item[(ii)] A standard closed graph argument shows that
$S_f^p$ (defined in~\eqref{eqn:Sf}) is an operator (see e.g. \cite[p.~52, Lemma~1]{die-uhl-J} for a proof of the case $p=1$), hence
$$
	\|f\|_{\mathscr{D}_p(\mu,X)}:=\sup_{x^*\in B_{X^*}}
	\left(
	\int_\Omega 
	|\langle f,x^*\rangle|^p \, d\mu
	\right)^{1/p}<\infty.
$$
In particular, the family of real-valued functions
$$ 
	Z_f:=\big\{\langle f,x^*\rangle: \, x^*\in B_{X^*}\big\},
$$
is uniformly integrable in~$L^1(\mu)$ whenever $p>1$. 
\item[(iii)] For each $g\in L^{p'}(\mu)$ the product $gf:\Omega \to X$
is Dunford integrable and 
$$
	T^p_f(g)=\int_\Omega g f  \, d\mu.
$$
In particular, $T^p_f(\chi_A)=\int_A f\, d\mu$ for all $A\in \Sigma$.
\item[(iv)] If $f$ is {\em $p$-Pettis integrable}, then
$T^p_f$ takes values in~$X$ and $gf$ is Pettis integrable for every $g\in L^{p'}(\mu)$.
Moreover, in this case $Z_f$ is uniformly integrable in $L^1(\mu)$ even for $p=1$
(see e.g. \cite[Corollary~4.1]{Musial}).
\item[(v)] By Schauder's theorem, the compactness of~$T^p_f$ is equivalent to that of~$S^p_f$.
\end{enumerate}
\end{remark}

In general, scalarly measurable bounded functions might not be Pettis integrable. This is an 
interesting phenomenon ocurring in the Pettis integral theory of non strongly measurable functions.
The space $X$ is said to have the {\em Pettis Integral Property} 
with respect to~$\mu$ (shortly $\mu$-PIP) if every scalarly bounded and scalarly measurable function $f:\Omega \to X$ is Pettis integrable. 
The $\mu$-PIP is equivalent to the following (apparently stronger) condition: a function 
$f:\Omega \to X$ is Pettis integrable if (and only if) it is Dunford integrable and $Z_f$ is uniformly integrable in $L^1(\mu)$.
The space $X$ is said to have the {\em Pettis Integral Property} (PIP) if it has the $\mu$-PIP
for any complete probability space $(\Omega,\Sigma,\mu)$. The class of Banach spaces having the PIP is rather wide and 
includes, for instance, all spaces having Corson's property~(C), all spaces having Mazur's property and all spaces which
are weakly measure-compact. In particular,
every weakly compactly generated space has the PIP. We refer the reader to \cite[Chapter~7]{Musial} and \cite[Section~8]{mus3}
for more information on the PIP. The following connection is immediate:

\begin{corollary}\label{corollary:PIP}
The following statements are equivalent:
\begin{enumerate}
\item[(i)] $X$ has the $\mu$-PIP; 
\item[(ii)] for some/any $1<p<\infty$, every $p$-Dunford integrable function $f:\Omega \to X$ is $p$-Pettis integrable.
\end{enumerate}
\end{corollary}
\begin{proof}
(ii)$\impli$(i): Note that any scalarly bounded and scalarly measurable function $f:\Omega \to X$
is $p$-Dunford integrable, for any $1\leq p<\infty$.

(i)$\impli$(ii): Let $f:\Omega \to X$ be a $p$-Dunford integrable function for some $1<p<\infty$. According
to Remark~\ref{remark:basic-pDunford}(ii), $Z_f$ is uniformly integrable in~$L^1(\mu)$ and so
the $\mu$-PIP of~$X$ ensures that $f$ is Pettis integrable.
\end{proof}

For any strongly measurable function $f:\Omega \to X$ there is 
a separable subspace $Y \sub X$ such that $f(\omega)\in Y$ for $\mu$-a.e. $\omega \in \Omega$.
Since separable Banach spaces have the PIP, from Corollary~\ref{corollary:PIP} we get the classical result mentioned in the introduction:

\begin{corollary}[Pettis]\label{corollary:DunfordImpliesPettis}
Suppose $1<p<\infty$. Then every strongly measurable $p$-Dunford integrable function $f:\Omega \to X$ is $p$-Pettis integrable.
\end{corollary}

\section{Compactness of the Dunford operator}\label{section:compact}

We first revisit, with an easier proof, Pettis' example (see \cite[9.3]{Pettis}) of a strongly measurable $2$-Pettis integrable
function having non-compact Dunford operator.

\begin{example}[Pettis]\label{example:Pettis}
Let $(f_n)$ be an orthonormal system in $L^2[0,1]$ and let us consider the strongly measurable function
$$
	f:[0,1]\to L^2[0,1], \quad f(t):=\sum_{n=1}^\infty 2^n f_n \cdot \chi_{I_n}(t),
$$
where $I_n:=(1/2^n,1/2^n+1/4^n)$ for all $n\in \Nat$ (so that the $I_n$'s are pairwise
disjoint). For each $g\in L^2[0,1]^*=L^2[0,1]$ we have
$$
	\Big(\int_0^1 |\langle f,g\rangle(t)|^2 \, dt\Big)^{1/2}=
	\Big(\sum_{n=1}^\infty |\langle f_n,g \rangle|^2\Big)^{1/2} \leq \|g\|_{L^2[0,1]},
$$
and so $\langle f,g\rangle\in L^2[0,1]$. Thus, $f$ is $2$-Dunford integrable and hence 
$2$-Pettis integrable (apply Corollary~\ref{corollary:DunfordImpliesPettis}). 
Let us check that the operator $S^2_f:L^2[0,1]\to L^2[0,1]$ is not compact. Indeed, observe that
$S^2_f(f_n)=\langle f,f_n\rangle=2^n\chi_{I_n}$ for all $n\in \Nat$.
Since $(2^n\chi_{I_n})$ is an orthonormal system in $L^2[0,1]$, it does not admit any norm
convergent subsequence. Therefore, $S^2_f$ is not compact.
\end{example}

The previous construction can be generalized, as we show in Example~\ref{example:C} below. By a {\em K{\"o}the function space} over~$\mu$ we mean an order 
ideal $Z$ of~$L^1(\mu)$ containing all simple functions which is equipped with a complete lattice norm. The space $Z$
is said to be {\em $q$-convex}, for a given $1\leq q<\infty$, if there is a constant $C>0$ such that
$$
	\Big\|  \Big( \sum_{i=1}^n  |z_i|^q \Big)^{1/q} \Big\|_{Z} \leq
	C \Big( \sum_{i=1}^n  \|z_i\|^q_{Z} \Big)^{1/q},
$$
for every $n\in \Nat$ and $z_1,\dots,z_n\in Z$. For instance, $L^q(\mu)$ is a $q$-convex K{\"o}the function space over~$\mu$. 
We refer the reader to \cite{lin-tza-2} for more information on $q$-convexity and related notions in Banach lattices.

\begin{example}\label{example:C}
Suppose that there is 
an infinite sequence $(A_i)$ of pairwise disjoint elements of~$\Sigma$ with $\mu(A_i)>0$ for all $i\in \Nat$. 
Let $1 < p < \infty$. Let $Z$ be a K{\"o}the function space over~$\mu$ which is $p'$-convex and order continuous.
Then there is a strongly measurable $p$-Pettis integrable function $\phi:\Omega\to Z$ such that its Dunford operator 
$T^p_\phi:L^{p'}(\mu)\to Z$ is not compact.
\end{example}
\begin{proof} 
Since $Z$ is order continuous, its topological dual $Z^*$ coincides with the K\"{o}the dual of~$Z$, i.e. the set
of all $g\in L^1(\mu)$ such that $fg\in L^1(\mu)$ for all $f\in Z$, the duality being given by $\langle f,g\rangle=\int_\Omega fg\, d\mu$
(see e.g. \cite[p.~29]{lin-tza-2}). 
On the other hand, since $Z$ is $p'$-convex, $Z^{*}$ is $p$-concave (see e.g. \cite[Proposition~1.d.4]{lin-tza-2}), i.e. there is a constant $M>0$
such that
\begin{equation}\label{eqn:pconcave}
	\Big( \sum_{i=1}^n  \|h_i\|^p_{Z^{*}} \Big)^{1/p} \leq
    M \Big\|  \Big( \sum_{i=1}^n  |h_i|^p \Big)^{1/p} \Big\|_{Z^*},
\end{equation}
for every $n\in \Nat$ and $h_1,\dots,h_n\in Z^{*}$. 
For each $i\in \Nat$, we fix $f_i\in S_Z$ such that $f_i=f_i\chi_{A_i}$
(e.g. $f_i=\|\chi_{A_i}\|_{Z}^{-1}\chi_{A_i}$) and  
we choose $g_i\in B_{Z^*}$ such that $\langle f_i, g_i \rangle=1$. 
Note that $g_i\chi_{A_i}\in B_{Z^{*}}$ also satisfies this equality, 
so we can assume without loss of generality that $g_i=g_i\chi_{A_i}$.
Define a strongly measurable function
$\phi:\Omega \to Z$ by
$$
	\phi(\omega):= \sum_{i=1}^\infty \mu(A_i)^{-1/p} f_i \, \cdot \, \chi_{A_i}(\omega).
$$

Let us check first that $\phi$ is $p$-Pettis integrable. To this end, take any $g\in Z^*$ and $n\in \Nat$. Observe that
\begin{multline*}
		\int_{\bigcup_{i=1}^n A_i} |\langle \phi,g \rangle|^p \, d\mu
		=
		\sum_{i=1}^n \big| \mu(A_i)^{-1/p} \langle f_i, g \rangle \big|^p \, \mu(A_i)
		=
		\sum_{i=1}^n \big| \langle f_i, g \rangle \big|^p \\
    		=
		\sum_{i=1}^n \Big| \int_\Omega f_i g \, d\mu \Big|^p
		=
		\sum_{i=1}^n \big| \langle f_i, g\chi_{A_i} \rangle \big|^p
		\leq
		\sum_{i=1}^n  \|g \chi_{A_i}\|^p_{Z^*} \\
		\stackrel{\eqref{eqn:pconcave}}{\leq}
		M^p \Big\|  \Big( \sum_{i=1}^n  |g \chi_{A_i}|^p \Big)^{1/p} \Big\|^p_{Z^{*}}
		=
		M^p\Big\| \sum_{i=1}^n  |g \chi_{A_i}|  \Big\|^p_{Z^*}
		\leq
		M^p\| g \|^p_{Z^{*}}.
\end{multline*}
As $n\in \Nat$ is arbitrary, it follows that $\int_{\Omega} |\langle \phi,g \rangle|^p \, d\mu<\infty$. 
Thus, $\phi$ is $p$-Dunford integrable. Since $\phi$ is strongly measurable and $p>1$, 
we conclude that $\phi$ is $p$-Pettis integrable (Corollary~\ref{corollary:DunfordImpliesPettis}).

To finish the proof we will check that the operator $S^p_\phi: Z^* \to L^{p}(\mu)$ is not compact. 
Indeed, since $\langle f_i,g_i\rangle=1$ for all $i\in \Nat$ and 
$$
	\langle f_j,g_i\rangle=\int_\Omega f_jg_i\, d\mu=\int_{A_j\cap A_i} f_jg_i\, d\mu=0 \quad \mbox{whenever }i\neq j,
$$ 
we have
$S^p_\phi(g_i)=\mu(A_i)^{-1/p} \chi_{A_i}$ for all $i\in \Nat$. 
Thus, $(S^p_\phi(g_i))$ is a sequence of norm one vectors in $L^p(\mu)$ having pairwise disjoint supports,
so it does not have norm convergent subsequences. Therefore, $S^p_\phi$ is not compact. 
\end{proof}

The proof of the following result is similar to that of the case $p=1$ (see e.g. \cite[Theorem~9.1]{Musial}) and is included
for the sake of completeness. 

\begin{theorem}\label{theorem:compactVSapprox}
Let $f:\Omega \to X$ be a $p$-Pettis integrable function. The following statements are equivalent:
\begin{enumerate}
\item[(i)] $T^p_f$ is compact;
\item[(ii)] for every $\varepsilon>0$ there is a simple function $h:\Omega \to X$ such that 
$$
	\|f-h\|_{\mathscr{D}_p(\mu,X)}\leq \varepsilon.
$$ 
\end{enumerate} 
\end{theorem}
\begin{proof} (ii)$\impli$(i): Note that for any simple function $h:\Omega \to X$ the operator
$S^p_h$ has finite rank (hence it is compact) and
$$
	\|S^p_f-S^p_h\|=\|S^p_{f-h}\|=\|f-h\|_{\mathscr{D}_p(\mu,X)}.
$$
From (ii) and the previous comments it follows at once that $S^p_f$ is compact.

(i)$\impli$(ii): Let $\Pi$ be the collection of all partitions of~$\Omega$ into finitely many measurable sets, which becomes
a directed set when ordered by refinement. For each $\mathcal{P}\in \Pi$, we define an operator $U_\mathcal{P}:L^p(\mu)\to L^p(\mu)$ by
$$
	U_\mathcal{P}(g):=\sum_{A\in \mathcal{P}}\frac{\int_A g \, d\mu}{\mu(A)}\cdot \chi_A,
$$
(with the convention $\frac{0}{0}=0$). We have $\sup_{\mathcal{P}\in \Pi}\|U_\mathcal{P}\|\leq 1$ and 
$$
	\lim_{\mathcal{P}}\|U_\mathcal{P}(g)- g\|_{L^p(\mu)}= 0 \quad 
	\mbox{for every }g\in L^p(\mu)
$$ 
(the proof of the case $p=1$ given in \cite[pp.~67--68, Lemma~1]{die-uhl-J}
can be easily adapted to the general case). Therefore, for any relatively norm compact set $K \sub L^p(\mu)$ we have
$$
	\lim_{\mathcal{P}}\sup_{g\in K}\|U_\mathcal{P}(g)- g\|_{L^p(\mu)}= 0.
$$

Fix $\epsilon>0$. Since $S^p_f$ is compact, the set $K:=S^p_f(B_{X^*})$ is relatively norm compact in~$L^p(\mu)$ and, therefore,
there is $\mathcal{P}_0\in \Pi$ such that
\begin{equation}\label{eqn:Upi}
	\sup_{x^*\in B_{X^*}}\|U_\mathcal{P}(\langle f,x^*\rangle)- \langle f,x^*\rangle\|_{L^p(\mu)}\leq \epsilon
	\quad \mbox{for any }\mathcal{P}\in \Pi \mbox{ finer than }\mathcal{P}_0.
\end{equation}
Since $f$ is Pettis integrable, for each $\mathcal{P}\in \Pi$ the simple function
$$
	h_\mathcal{P}:\Omega \to X, \quad h_\mathcal{P}:=\sum_{A\in \mathcal{P}}\frac{1}{\mu(A)}\int_A f \, d\mu \cdot \chi_A,
$$
satisfies $\langle h_\mathcal{P},x^*\rangle =U_\mathcal{P}(\langle f,x^*\rangle)$ for all $x^*\in X^*$. Hence~\eqref{eqn:Upi} reads as
$$
	\|h_\mathcal{P}-f\|_{\mathscr{D}_p(\mu,X)}\leq \epsilon
	\quad
	\mbox{for any }\mathcal{P}\in \Pi \mbox{ finer than }\mathcal{P}_0.
$$ 
This finishes the proof.
\end{proof}

Any $p$-Dunford integrable function $f:\Omega\to X$ is Dunford integrable and we have a factorization
$$
	\xymatrix@R=3pc@C=3pc{L^\infty(\mu)
	\ar[r]^{T^1_f} \ar[d]_{i} & X^{**}\\
	L^{p'}(\mu)  \ar@{->}[ur]_{T^p_f}  & \\
	}
$$
where $i$ is the inclusion operator. Our next result clarifies the relationship between 
the compactness of $T_f^p$ and that of $T_f^1$.

\begin{theorem}\label{theorem:CompactDunfordOperator}
Let $f:\Omega \to X$ be a $p$-Dunford integrable function. The following statements are equivalent:
\begin{itemize}
\item[(i)] $T_f^p$ is compact;
\item[(ii)] $T_f^1$ is compact and 
$$
	Z_f^p=\big\{|\langle f,x^*\rangle|^p: \, x^*\in B_{X^*}\big\},
$$
is relatively norm compact in $L^1(\mu)$;
\item[(iii)] $T_f^1$ is compact and $Z_{f}^p$ is uniformly integrable in $L^1(\mu)$.
\end{itemize}
\end{theorem}

The proof of Theorem~\ref{theorem:CompactDunfordOperator} requires a couple of lemmas.

\begin{lemma}\label{lemma:sequential}
Let $f:\Omega \to X$ be a $p$-Dunford integrable function. If $T^p_f$ is compact, then for every sequence $(x_n^*)$ in~$B_{X^*}$ there exist a subsequence $(x_{n_k}^*)$
and $x^*\in B_{X^*}$ such that $\langle f,x_{n_k}^* \rangle \to \langle f,x^*\rangle$ $\mu$-a.e.
\end{lemma}
\begin{proof}
The set $S_f^p(B_{X^*})$ is relatively norm compact in~$L^p(\mu)$. Hence there exist a subsequence $(x_{n_k}^*)$ and $h\in L^p(\mu)$ such that
$\|\langle f,x_{n_k}^* \rangle-h\|_{L^p(\mu)}\to 0$. By passing to a further subsequence, not relabeled, we can assume that $\langle f,x_{n_k}^* \rangle \to h$ $\mu$-a.e.
Let $x^*\in B_{X^*}$ be a $w^*$-cluster point of $(x_{n_k}^*)$. Then $\langle f(\omega),x^* \rangle$ is a cluster
point of the sequence of real numbers $(\langle f(\omega),x_{n_k}^* \rangle)$ for each $\omega \in \Omega$. It follows
that $h= \langle f,x^* \rangle$ $\mu$-a.e. and $\langle f,x_{n_k}^* \rangle \to \langle f,x^*\rangle$ $\mu$-a.e.
\end{proof}

We include a proof of the following well-known fact since we did not find a suitable reference for it.

\begin{lemma}\label{lemma:elevar}
The map $L^p(\mu)\to L^1(\mu)$ given by $h\mapsto |h|^p$ is norm-to-norm continuous.
\end{lemma}
\begin{proof}
Let $(h_n)$ be a norm convergent sequence in $L^p(\mu)$, with limit $h\in L^p(\mu)$. 
If $p=1$, then we have
$$
	\big\||h_n|-|h|\big\|_{L^1(\mu)}=\int_\Omega \big||h_n|-|h|\big| \, d\mu\leq \int_\Omega |h_n-h| \, d\mu \to 0. 
$$
Suppose now that $p>1$. Fix a constant $C>0$ such that $\|h_n\|_{L^p(\mu)}\leq C$ for all $n\in \Nat$.
Bearing in mind that
$$
	|a^p-b^p|\leq p\cdot |a^{p-1}+b^{p-1}| \cdot |a-b| \quad
	\mbox{for every }a,b\geq 0,
$$
H\"{o}lder's inequality and the triangle
inequality in~$L^{p'}(\mu)$, we obtain
\begin{multline*}
	\int_\Omega  \big||h_n|^p - |h|^p \big|  \, d \mu \le
	\int_\Omega  \, p\cdot \big| |h_n|^{p-1} + |h|^{p-1} \big| \cdot
	\big| |h_n| - |h| \big| \, d \mu
	\\
	\le 
	p \cdot \Big\||h_n|^{p-1} + |h|^{p-1} \Big\|_{L^{p'}(\mu)}
	\cdot  \Big( \int_\Omega 
	\big| |h_n| - |h| \big|^{p} \, d \mu \Big)^{1/p}
	\\
	\le
	2p \, C^{p/p'}
	\cdot  \Big( \int_\Omega \big| h_n - h \big|^{p} \, d \mu \Big)^{1/p} \to 0.
\end{multline*}
This proves that $|h_n|^p\to |h|^p$ in~$L^1(\mu)$.
\end{proof}

\begin{proof}[Proof of Theorem~\ref{theorem:CompactDunfordOperator}] 

(i)$\impli$(ii): Clearly, $T^1_f$ is compact as it factors through~$T^p_f$. On the other hand,
$S^p_{f}(B_{X^*}) = \{  \langle f, x^* \rangle: \, x^* \in B_{X^*} \big\}$
is relatively norm compact in $L^p(\mu)$. Therefore, an appeal to Lemma~\ref{lemma:elevar} ensures that
$Z_f^p$ is relatively norm compact in~$L^1(\mu)$.

(ii)$\impli$(iii): Obvious.

(iii)$\impli$(i): We will check that $S_f^p$ is compact. Let $(x_n^*)$ be a sequence in~$B_{X^*}$. By Lemma~\ref{lemma:sequential} (applied to $T_f^1$),
there exist a subsequence $(x_{n_k}^*)$ and $x^*\in B_{X^*}$ such that $\langle f,x^*_{n_k} \rangle \to \langle f,x^*\rangle$ $\mu$-a.e.
For each $k\in \Nat$, define $h_k:=|\langle f,x_{n_k}^*-x^*\rangle|^p \in L^1(\mu)$. Then $h_k\to 0$ $\mu$-a.e.
and $(h_k)$ is uniformly integrable in~$L^1(\mu)$ (bear in mind that $2^{-p}\cdot h_k\in Z^p_f$ for all $k\in \Nat$).
We can apply Vitali's convergence theorem to conclude that $\|h_k\|_{L^1(\mu)}\to 0$, that is,
$S^p_f(x_{n_k}^*)=\langle f,x_{n_k}^*\rangle \to S^p_f(x^*)=\langle f,x^*\rangle$ in the norm topology of~$L^p(\mu)$.
This proves that $S_f^p$ is compact. 
\end{proof}

Let $f:\Omega \to X$ be a Dunford integrable function  such that
$$
	Z_f=\{\langle f,x^* \rangle: \, x^*\in B_{X^*}\},
$$
is uniformly integrable in~$L^1(\mu)$. The Dunford operator $T^1_f$ is known to be compact
in many cases, for instance, if $\mu$ is perfect (e.g. a Radon measure) or if $X\not \supseteq \ell^1(\aleph_1)$, 
see e.g. \cite[Theorem~4-1-6]{tal}. In particular,
the compactness of $T^1_f$ is guaranteed if either $\mu$ is the Lebesgue measure on~$[0,1]$ or 
$X$ is weakly compactly generated.
 
Let us say that $X$ has the {\em Dunford Compactness Property} with respect to~$\mu$ (shortly $\mu$-DCP) if
$T_f^1$ is compact whenever $f:\Omega \to X$ is Dunford integrable and $Z_f$ is uniformly integrable in~$L^1(\mu)$.
For more information on this property, we refer to \cite[Chapter~9]{Musial}, \cite[Section~5]{mus3} and \cite[Chapter~4]{tal}.

\begin{corollary}\label{cor:DCP}
Suppose $X$ has the $\mu$-DCP. Let $f:\Omega \to X$ be a $p$-Dunford integrable function. The following statements are equivalent:
\begin{itemize}
\item[(i)] $T_f^p$ is compact;
\item[(ii)]$Z_{f}^p$ is uniformly integrable in $L^1(\mu)$.
\end{itemize}
\end{corollary}
\begin{proof}
Note that (ii) implies that $Z_f$ is uniformly integrable in $L^1(\mu)$, because H\"{o}lder's inequality yields
$$
	\int_A|\langle f,x^*\rangle| \, d\mu \leq \Big(\int_A |\langle f,x^*\rangle|^p \, d\mu\Big)^{1/p}
	\quad\mbox{for every }A\in \Sigma \mbox{ and }x^*\in B_{X^*}.
$$
The result now follows at once from Theorem~\ref{theorem:CompactDunfordOperator}. 
\end{proof}

Bearing in mind that separable Banach spaces have the $\mu$-DCP, we get:

\begin{corollary}\label{cor:CompactSM}
Let  $f:\Omega \to X$ be a strongly measurable $p$-Dunford integrable function. The following statements are equivalent:
\begin{itemize}
\item[(i)] $T_f^p$ is compact;
\item[(ii)]$Z_{f}^p$ is uniformly integrable in $L^1(\mu)$.
\end{itemize}
\end{corollary}

\section{$p$-summing operators and $p$-Dunford integrable functions}\label{section:summing}

Vector measures of bounded $p$-variation play a relevant role in the duality theory of
Lebesgue-Bochner spaces, see \cite[p.~115]{die-uhl-J} and \cite[\S13]{din-J}. We begin this section
by giving some auxiliary results on the $p$-variation of a vector measure which will be helpful when studying
the composition of $p$-Dunford integrable functions with $p$-summing operators.

We denote by $\Pi$ the collection of all finite partitions of $\Omega$ into measurable sets.
Given a Banach space $Z$, we write $V(\mu,Z)$ for the set of all finitely additive vector measures $\nu:\Sigma\to Z$
such that $\nu(A)=0$ whenever $\mu(A)=0$.

\begin{definition}\label{definition:pVariation}
Let $Z$ be a Banach space. The {\em total $p$-variation} of~$\nu\in V(\mu,Z)$ is defined by
\begin{eqnarray*}
	|\nu|_p(\Omega)&:=&\sup\left\{
	\left(\sum_{A\in \mathcal{P}} \frac{\|\nu(A)\|_Z^{p}}{\mu(A)^{p-1}}\right)^{1/p}: \, \mathcal{P} \in \Pi
	\right\}\\
	&=&
	\sup\left\{\sum_{A\in \mathcal{P}} |\alpha_A|\big\|\nu(A)\big\|_Z: \, \mathcal{P} \in \Pi, \, \sum_{A\in \mathcal{P}} \alpha_A\chi_{A}\in B_{L^{p'}(\mu)}\right\} \in [0,\infty],
\end{eqnarray*}
(with the convention $\frac{0}{0^{p-1}}=0$). We say that $\nu$ has {\em bounded $p$-variation} if $|\nu|_p(\Omega)$ is finite.
\end{definition}

A function $f:\Omega \to X$ is said to be {\em $p$-Bochner integrable} if it is strongly measurable and the real-valued
function $\|f(\cdot)\|_X$ belongs to~$L^p(\mu)$. In this case, $f$ is $p$-Pettis integrable and it is known that
the countably additive vector measure
$$
	\nu_f:\Sigma\to X, \quad \nu_f(A)=\int_A f \, d\mu,
$$
has bounded $p$-variation and 
$$
	\big|\nu_f\big|_p(\Omega)=\Big(\int_\Omega \|f(\cdot)\|^p_X\, d\mu\Big)^{1/p}=:\|f\|_{L^p(\mu,X)}.
$$ 
In Theorem~\ref{theorem:pVariation} below we show that this equality holds for 
any strongly measurable $p$-Dunford integrable function~$f:\Omega \to X$. 
Of course, in this more general case the total $p$-variation of the corresponding measure~$\nu_f$ (which belongs to $V(\mu,X^{**})$) can be infinite.
We should point out that Theorem~\ref{theorem:pVariation} for $p=1$ is a particular case of 
a well-known result (see e.g. \cite[Theorem~4.1 and Remark~4.2]{Musial}). 

The following lemma is folklore (see e.g. \cite[p.~42, Corollary~3]{die-uhl-J}).

\begin{lemma}\label{lem:countably-valued}
A function $f:\Omega \to X$ is strongly measurable if and only if for every
$\epsilon>0$ there exist a sequence $(A_{n})$ of pairwise disjoint measurable sets 
and a sequence $(x_{n})$ in~$X$ such that
the function $g:=\sum_{n=1}^{\infty}x_{n}\chi_{A_{n}}$ satisfies
$\|f(\omega)-g(\omega)\|_X \leq \epsilon$ for $\mu$-a.e. $\omega\in \Omega$. 
\end{lemma}

\begin{theorem}\label{theorem:pVariation}
Let $f:\Omega \to X$ be a strongly measurable $p$-Dunford integrable function. Then
\begin{equation}\label{eqn:pFormula}
	\big|\nu_f\big|_p(\Omega)=\Big(\int_\Omega \|f(\cdot)\|_X^p \, d\mu\Big)^{1/p}.
\end{equation}
In particular, $f$ is $p$-Bochner integrable if and only if $\nu_f$ has bounded $p$-variation.
\end{theorem}
\begin{proof}
Let us prove the inequality ``$\leq$'' in~\eqref{eqn:pFormula}. Of course, we may (and do) assume that $f$ is $p$-Bochner integrable
since otherwise the inequality is obvious.
Take $\mathcal{P}\in \Pi$ and $g\in B_{L^{p'}(\mu)}$ of the form $g=\sum_{A\in \mathcal{P}} \alpha_A \chi_{A}$, where $\alpha_A\in \mathbb R$. 
Fix $\varepsilon>0$. Then for each $A\in \mathcal{P}$ there exists $x_A^*\in S_{X^*}$ such that 
\begin{equation}\label{eqn:HB}
	|\alpha_A|\big\|\nu_f(A)\big\|_{X^{**}} \leq |\alpha_A| \langle \nu_f(A),x^*_A\rangle +\frac{\varepsilon}{\# \mathcal{P}}
	=|\alpha_A|  \int_{A} \langle f,x^*_A \rangle \, d\mu +\frac{\varepsilon}{\# \mathcal{P}},
\end{equation}
(where $\# \mathcal{P}$ stands for the cardinality of~$\mathcal{P}$). 
Let $\tilde{g}:\Omega \to X^*$ be the simple function defined by $\tilde{g}:=\sum_{A\in \mathcal{P}} |\alpha_A|  x^*_A \chi_{A}$.
By applying H\"{o}lder's inequality to the non-negative measurable functions
$\|f(\cdot)\|_{X}$ and $\|\tilde{g}(\cdot)\|_{X^*}=|g(\cdot)|$, we have:
$$
	\sum_{A\in \mathcal{P}} |\alpha_A| \big\|\nu_f(A)\big\|_{X^{**}} \stackrel{\eqref{eqn:HB}}{\leq} 
	\sum_{A\in \mathcal{P}} |\alpha_A| \int_{A} \langle f,x^*_A\rangle d\mu + \varepsilon
$$
$$
	=\int_{\Omega}   \left\langle f(\cdot), \tilde{g}(\cdot)\right\rangle d\mu + \varepsilon
	\leq\int_{\Omega}   \|f(\cdot)\|_{X} \|\tilde{g}(\cdot)\|_{X^*} \, d\mu + \varepsilon
$$
$$
	\leq \Big(\int_{\Omega} \|f(\cdot)\|^p_{X} \, d\mu\Big)^{1/p} \|g\|_{L^{p'}(\mu)}+ \varepsilon
	\leq \Big(\int_{\Omega} \|f(\cdot)\|^p_{X} \, d\mu\Big)^{1/p} + \varepsilon.
$$
Since both $\varepsilon>0$ and $g$ are arbitrary we obtain 
\begin{equation}\label{eqn:first}
	\big|\nu_f\big|_p(\Omega)\leq \Big(\int_\Omega \|f(\cdot)\|_X^p \, d\mu\Big)^{1/p}.
\end{equation}

In order to prove the inequality ``$\geq$'' in~\eqref{eqn:pFormula} we can assume that $\big|\nu_f\big|_p(\Omega)$ is finite. We begin with the
following:

{\em Particular case}. Suppose that $f=\sum_{n=1}^\infty x_n \chi_{A_n}$, where $(A_n)$ is a sequence of pairwise disjoint measurable sets 
and $(x_{n})$ is a sequence in~$X$. 
The inequality ``$\geq$'' in~\eqref{eqn:pFormula} is obvious if $f=0$ $\mu$-a.e. Otherwise,
we have $\sum_{n=1}^N \|x_n\|^p\mu(A_n)>0$ for large enough $N\in \Nat$. Fix such an~$N$ and define
$$
	\alpha_i:=
	\begin{cases}
	\|x_i\|^{p-1}\Big(\sum_{n=1}^N \|x_n\|^p \mu(A_n)\Big)^{-1/p'} & \text{if $x_i\neq 0$} \\
	0 & \text{if $x_i= 0$}
	\end{cases}
$$
for all $i=1,\dots,n$. It is easy to see that $g=\sum_{i=1}^N \alpha_i\chi_{A_i}$ belongs to~$S_{L^{p'}(\mu)}$.
Therefore
$$
	\big|\nu_f\big|_p(\Omega)\geq \sum_{i=1}^N |\alpha_i|\big\|\nu_f(A_i)\big\|_{X^{**}}=\sum_{i=1}^N |\alpha_i|\|x_i\|\mu(A_i)=
	\Big(\sum_{n=1}^N  \|x_n\|^p\mu(A_n)\Big)^{1/p},
$$
and by taking limits when $N\to\infty$ we obtain $\big|\nu_f\big|_p(\Omega)\geq \big(\int_{\Omega}   \|f(\cdot)\|^p_{X} \, d\mu\big)^{1/p}$.

{\em General case}. Fix $\varepsilon>0$. Lemma~\ref{lem:countably-valued} ensures the existence of a sequence $(A_{n})$ of pairwise disjoint measurable sets 
and a sequence $(x_{n})$ in~$X$ such that the function $g:=\sum_{n=1}^{\infty}x_{n}\chi_{A_{n}}$ satisfies
$\|f(\cdot)-g(\cdot)\|_X \leq \epsilon$ $\mu$-a.e. Since the functions $f$ and $f-g$ are $p$-Dunford integrable,
so is~$g$, with $\nu_{g}=\nu_f-\nu_{f-g}$. On one hand, \eqref{eqn:first} applied to~$f-g$ yields
\begin{equation}\label{eqn:second}
	\big|\nu_g\big|_p(\Omega)\leq \big|\nu_{f-g}\big|_p(\Omega)+\big|\nu_f\big|_p(\Omega)
	\leq \varepsilon+\big|\nu_f\big|_p(\Omega).
\end{equation}
On the other hand, we can apply the {\em Particular case} to $g$, so that
$$ 
	\Big(\int_{\Omega}   \|g(\cdot)\|^p_{X} \, d\mu\Big)^{1/p}\leq \big|\nu_g\big|_p(\Omega)
	\stackrel{\eqref{eqn:second}}{\leq} \varepsilon+\big|\nu_f\big|_p(\Omega).
$$
In particular, $g$ is $p$-Bochner integrable, and so does $f=(f-g)+g$. Moreover, 
\begin{eqnarray*}
	\Big(\int_{\Omega}   \|f(\cdot)\|^p_{X} \, d\mu\Big)^{1/p}&\leq& \Big(\int_{\Omega}   \|f(\cdot)-g(\cdot)\|^p_{X} \, d\mu\Big)^{1/p}+
	\Big(\int_{\Omega}   \|g(\cdot)\|^p_{X} \, d\mu\Big)^{1/p}\\
	&\leq&2\varepsilon+\big|\nu_f\big|_p(\Omega).
\end{eqnarray*}
As $\varepsilon>0$ is arbitrary, $\big|\nu_f\big|_p(\Omega)\geq \big(\int_{\Omega}   \|f(\cdot)\|^p_{X} \, d\mu\big)^{1/p}$.
The proof is finished.
\end{proof}

\begin{definition}\label{definition:pSemiVariation}
Let $Z$ be a Banach space. The {\em total $p$-semivariation} of~$\nu\in V(\mu,Z)$ is defined by
$$
	\|\nu\|_p(\Omega):=\sup\big\{|\langle \nu,z^*\rangle|_p(\Omega): \, z^*\in B_{Z^*}\big\},
$$
where $\langle \nu,z^*\rangle \in V(\mu,\mathbb{R})$ stands for the composition of $z^*$ and~$\nu$.  
We say that $\nu$ has {\em bounded $p$-semivariation} if $\|\nu\|_p(\Omega)$ is finite.
\end{definition}

\begin{remark}\label{remark:NormingSemivariation}
Let $Z$ be a Banach space and $\nu\in V(\mu,Z)$. If $\Delta$ is any $w^*$-dense subset of~$B_{Z^*}$, then
$\|\nu\|_p(\Omega)=\sup\big\{|\langle \nu,z^*\rangle|_p(\Omega): z^*\in \Delta\big\}$.
\end{remark}
\begin{proof}
Write $\gamma:=\sup\big\{|\langle \nu,z^*\rangle|_p(\Omega): z^*\in \Delta\big\}$. Fix $\epsilon>0$ and
$z^*\in B_{Z^*}$. Take $\mathcal{P}\in \Pi$ and $g\in B_{L^{p'}(\mu)}$ of the form 
$g=\sum_{A\in \mathcal{P}} \alpha_A \chi_{A}$, where $\alpha_A\in \mathbb R$. Since $\Delta$ is $w^*$-dense in~$B_{Z^*}$,
there is $z_0^*\in \Delta$ such that
$$
	\sum_{A\in \mathcal{P}} |\alpha_A| \big|\langle \nu(A),z^* \rangle\big|\leq
	\sum_{A\in \mathcal{P}} |\alpha_A| \big|\langle \nu(A),z_0^* \rangle\big| + \epsilon
	\leq |\langle \nu,z_0^*\rangle|_p(\Omega)+\epsilon \leq \gamma + \epsilon.
$$
It follows that $|\langle \nu,z^*\rangle|_p(\Omega)\leq \gamma+\epsilon$. As $\epsilon$ and $z^*$ are arbitrary, $\|\nu\|_p(\Omega)=\gamma$.
\end{proof}

\begin{corollary}\label{corollary:semiDunford}
If $f:\Omega \to X$ is $p$-Dunford integrable, then $\nu_f\in V(\mu,X^{**})$ has bounded $p$-semivariation
and $\|\nu_f\|_p(\Omega)=\|f\|_{\mathscr D^p(\mu,X)}$.
\end{corollary}
\begin{proof} For each $x^*\in X^*$ the composition $\langle \nu_f,x^* \rangle$
is the indefinite integral of $\langle f,x^*\rangle \in L^p(\mu)$, hence
$|\langle \nu_f,x^* \rangle|_p(\Omega)=\|\langle f,x^*\rangle\|_{L^p(\mu)}$
(apply Theorem~\ref{theorem:pVariation} in the real-valued case to $\langle f,x^* \rangle$).
Now, Remark~\ref{remark:NormingSemivariation} applied to~$\nu_f$ with $Z:=X^{**}$ and $\Delta:=B_{X^*}$
(which is $w^*$-dense in~$B_{X^{***}}$ by Goldstine's theorem) ensures that 
\begin{multline*}
	\|\nu_f\|_p(\Omega)=\sup\big\{|\langle \nu_f,x^*\rangle|_p(\Omega): \, x^*\in B_{X^*}\big\} \\ =
	\sup\big\{\|\langle f,x^*\rangle\|_{L^p(\mu)}: \, x^*\in B_{X^*}\big\}=\|f\|_{\mathscr D^p(\mu,X)}.
\end{multline*}
\end{proof}

Throughout the rest of this section $Y$ is a Banach space. 
Recall that an operator $u:X\to Y$ is said to be {\em $p$-summing} if there exists a constant $K\geq 0$ such that
$$
	\Big(
	\sum_{i=1}^n \|u(x_i)\|_Y^p
	\Big)^{1/p}
	\leq K
	\sup\left\{
	\Big(
	\sum_{i=1}^n \big|\langle x_i,x^*\rangle \big|^p
	\Big)^{1/p}
	: \,
	x^*\in B_{X^*}
	\right\},
$$
for every $n\in \Nat$ and every $x_1,\ldots,x_n\in X$.
The least constant $K$ satisfying this condition is usually denoted by $\pi_p(u)$.

Our results on the composition of $p$-summing operators and $p$-Dunford integrable functions  (Theorems~\ref{theorem:pSumming-strong}
and~\ref{theorem:pSumming-scalar} below) will be obtained with the help of Theorem~\ref{theorem:pVariation} and the following two easy lemmas.
 
\begin{lemma}\label{lemma:Biadjoint}
Let $u:X\to Y$ be an operator and $f:\Omega \to X$ a $p$-Dunford integrable function. Then $u \circ f$ is $p$-Dunford integrable
and $\nu_{u\circ f}=u^{**}\circ \nu_f$.
\end{lemma} 
\begin{proof} The first statement is obvious. On the other hand, for every $A\in\Sigma$ and $y^*\in Y^*$ we have
\begin{multline*}
 	\left\langle \nu_{u\circ f}(A),y^*\right\rangle= 
 	\int_A  \langle u\circ f,y^*\rangle \, d\mu= \int_A  \langle f,u^*(y^*)\rangle \, d\mu
 	\\ =\langle   \nu_f(A),u^*(y^*) \rangle=
 	\big\langle (u^{**}\circ \nu_f)(A),y^*\big\rangle,
\end{multline*}
and so $\nu_{u\circ f}=u^{**}\circ \nu_f$.
\end{proof}
 
\begin{lemma}\label{lemma:pSumming-measures}
Let $u:X\to Y$ be a $p$-summing operator. If $\nu\in V(\mu,X)$ has bounded $p$-semivariation, then 
$u\circ\nu\in V(\mu,Y)$ has bounded $p$-variation and 
$$
	|u\circ \nu|_p(\Omega)\leq \pi_p(u) \|\nu\|_p(\Omega).
$$
\end{lemma}
\begin{proof}
For any $\mathcal{P}\in \Pi$ we have
$$
  \left(\sum_{A\in \mathcal{P}}  \frac {\|u(\nu(A))\|_Y^p}{\mu(A)^{p-1}}\right)^{1/p}=
  \left(\sum_{A\in \mathcal{P}}  \left\|u\left(\frac{\nu(A)}{\mu(A)^{1/p'}}\right)\right\|_Y^p\right)^{1/p}
$$
$$
	\leq \pi_p(u) \sup\left\{\left(\sum_{A\in \mathcal{P}} \left|\left\langle \frac{\nu(A)}{\mu(A)^{1/p'}},x^*\right\rangle\right|^p\right)^{1/p}: \, x^*\in B_{X^*}\right\}
$$
$$
	= \pi_p(u) \sup\left\{\left(\sum_{A\in \mathcal{P}} \frac{|\langle \nu(A),x^*\rangle|^p}{\mu(A)^{p-1}}\right)^{1/p}: \, x^*\in B_{X^*}\right\}
\leq \pi_p(u) \|\nu\|_p(\Omega).
$$
\end{proof}

\begin{theorem}\label{theorem:pSumming-strong}
Let $u:X\to Y$ be a $p$-summing operator and $f:\Omega\to X$ a $p$-Dunford integrable function such that
$u\circ f$ is strongly measurable. Then $u\circ f$ is $p$-Bochner integrable and
$\|u\circ f\|_{L^p(\mu,Y)}\leq \pi_p(u)\|f\|_{\mathscr{D}^p(\mu,X)}$. 
\end{theorem}
\begin{proof}
We know that $\nu_f\in V(\mu,X^{**})$ satisfies
$\|\nu_f\|_p(\Omega)=\|f\|_{\mathscr{D}^p(\mu,X)}<\infty$ (Corollary~\ref{corollary:semiDunford}).
On the other hand, the $p$-summability of~$u$ guarantees that of~$u^{**}$ (and
in fact $\pi_p(u)=\pi_p(u^{**})$), see e.g. \cite[Proposition~2.19]{die-alt}. 
So Lemma~\ref{lemma:pSumming-measures} applied to $u^{**}$ and $\nu_f$ ensures
that $u^{**}\circ \nu_f$ has bounded $p$-variation and, moreover, $\big|u^{**}\circ \nu_f\big|_p(\Omega) \leq \pi_p(u)\|\nu_f\|_p(\Omega)$. 

Observe that $u\circ f$ is strongly measurable, $p$-Dunford integrable and satisfies 
$\nu_{u\circ f}=u^{**}\circ \nu_f$ (Lemma~\ref{lemma:Biadjoint}). Therefore, Theorem~\ref{theorem:pVariation} applied to $u\circ f$ 
tells us that $u\circ f$ is $p$-Bochner integrable and  
$$
 	\|u\circ f\|_{L^p(\mu,Y)} = \big|\nu_{u\circ f}\big|_p(\Omega)=\big|u^{**}\circ \nu_f\big|_p(\Omega) 
 	\leq \pi_p(u) \|\nu_f\|_p(\Omega) =\pi_p(u)\|f\|_{\mathscr{D}^p(\mu,X)},
$$
as we wanted to prove.
\end{proof}

The celebrated Pietsch factorization theorem (see e.g. \cite[2.13]{die-alt}) states that for every $p$-summing operator
$u:X \to Y$ there is a regular Borel probability measure $\eta$ on $(B_{X^*},w^*)$
such that $u$ factors as 
$$
	\xymatrix{ X \ar[rr]^{u} \ar[d]_{i}  & & Y \\
	i(X)\ar[rr]^{j}&  & Z \ar[u]_{\hat{u}}}
$$
where: 
\begin{itemize}
\item $i$ is the canonical isometric embedding of~$X$ in the Banach space $C(B_{X^*})$
of all real-valued continuous functions on $(B_{X^*},w^*)$ (which is given by $i(x)(x^*):=x^*(x)$ for every $x\in X$ and $x^*\in B_{X^*}$);
\item $Z$ is a subspace of~$L^p(\eta)$;
\item $j$ is the restriction of the identity operator which maps each element of $C(B_{X^*})$ to its equivalence class in~$L^p(\eta)$; 
\item $\hat{u}$ is an operator. 
\end{itemize}
Note that $u(X)$ is separable whenever $Z$ is separable. Therefore,
in this case, for every scalarly measurable function $f:\Omega \to X$ the composition $u\circ f$ is strongly measurable. 

The previous discussion and Theorem~\ref{theorem:pSumming-strong} lead to Corollary~\ref{corollary:pSumming-strong} below. 
Recall that a probability measure $\eta$ (defined on a $\sigma$-algebra) is said to be {\em separable} if its measure algebra equipped with the Fr\'{e}chet-Nikod\'{y}m metric
is separable or, equivalently, if $L^p(\eta)$ is separable for some/all $1\leq p<\infty$.
We say that a compact Hausdorff topological space $K$ belongs to the class~$MS$ if every regular Borel
probability measure on~$K$ is separable.

\begin{corollary}\label{corollary:pSumming-strong}
Suppose $(B_{X^*},w^*)$ belongs to the class~$MS$. Let $u:X\to Y$ be a $p$-summing operator. If $f:\Omega\to X$ is $p$-Dunford integrable,
then $u\circ f$ is $p$-Bochner integrable and
$\|u\circ f\|_{L^p(\mu,Y)}\leq \pi_p(u)\|f\|_{\mathscr{D}^p(\mu,X)}$. 
\end{corollary}

The class $MS$ is rather wide and contains all compact spaces which are Eberlein, Rosenthal, weakly Radon-Nikod\'{y}m,
linearly ordered, etc. From the Banach space point of view, $(B_{X^*},w^*)$ belongs to the class~$MS$ whenever
$X$ is weakly countably determined, weakly precompactly generated (e.g. $X\not\supseteq \ell^1$), etc.
For more information on the class~$MS$ we refer to \cite[Section~3.1]{rod3} and the references therein.
Some recent works on this topic are \cite{avi-mar-ple,mar-ple:12,ple-sob}.

\begin{remark}\label{remark:SeparableRange}
One of the consequences of the aforementioned Pietsch theorem is that every $p$-summing operator is {\em completely continuous}, see e.g. \cite[Theorem~2.17]{die-alt}, i.e.
it maps weakly compact sets to norm compact sets. 
{\em If $X$ is weakly precompactly generated, then every completely continuous
operator $u:X\to Y$ has separable range.} Indeed, $X=\overline{{\rm span}}(G)$ for some weakly precompact set $G \sub X$ and so
$u(X) \sub \overline{{\rm span}}(u(G))$, where $u(G)$ is relatively norm compact (hence separable).
The same assertion holds if $X$ is weakly countably determined, but in this case
the proof is more involved, see \cite[Theorem~7.1]{tal1}.
\end{remark}

In the following result (which is stronger than Theorem~\ref{theorem:pSumming-strong}) we remove the strong measurability condition 
for $u\circ f$ at the cost of obtaining a weaker conclusion.

\begin{theorem}\label{theorem:pSumming-scalar}
Let $u:X\to Y$ be a $p$-summing operator and $f:\Omega\to X$ a $p$-Dunford integrable function. Then
$u\circ f$ is scalarly equivalent to a $p$-Bochner integrable function $g:\Omega \to Y$ and 
$\|g\|_{L^p(\mu,Y)}\leq \pi_p(u) \|f\|_{\mathscr{D}^p(\mu,X)}$.
\end{theorem}
\begin{proof}
Since $u$ is $p$-summing, it is also weakly compact (see e.g. \cite[Theorem~2.17]{die-alt})
and hence $Z:=\overline{u(X)}$ is weakly compactly generated. In particular, $(Z,w)$ is Lindel\"{o}f (see e.g. \cite[Theorem~14.31]{fab-ultimo})
and so measure-compact. A result
of Edgar (see~\cite[Proposition~5.4]{edgar1}, cf. \cite[Theorem 3-4-6]{tal})
ensures the existence of a strongly measurable function $g:\Omega\to Z$ such that
$u\circ f$ and $g$ are scalarly equivalent.

Since $u\circ f$ is $p$-Dunford integrable, the same holds for~$g$, with $\nu_g=\nu_{u\circ f}=u^{**}\circ \nu_f$ (Lemma~\ref{lemma:Biadjoint}). 
On the other hand, in a similar way as we did in Theorem~\ref{theorem:pSumming-strong}, we deduce that $\nu_g$ has bounded $p$-variation and
$\big|\nu_g\big|_p(\Omega)\leq \pi_p(u) \|f\|_{\mathscr{D}^p(\mu,X)}$. The result now follows from 
Theorem~\ref{theorem:pVariation} applied to~$g$.
\end{proof}

We finish this section by pointing out that in~\cite{pel-rue-san} a general approach is developed to obtain further results 
on the improvement of the integrability of a strongly measurable function by a summing operator.

\section{Testing $p$-Dunford integrability}\label{section:thick}

In this section we study the $p$-Dunford integrability of a function $f:\Omega \to X$
via the family of real-valued functions
$$
	Z_{f,\Gamma}=\big\{\langle f,x^* \rangle: \, x^*\in \Gamma\big\},
$$
for some $\Gamma \sub X^*$. To deal with our main result, Theorem~\ref{theorem:pDunford-thick}, 
we use some ideas from the proof of \cite[Theorem~9]{raj-rod}.

\begin{theorem}\label{theorem:pDunford-thick}
Suppose $X$ has the $\mu$-PIP. Let $f:\Omega \to X$
be a scalarly measurable function for which there is a $w^*$-thick set $\Gamma\sub X^*$ such that
$Z_{f,\Gamma}\sub L^p(\mu)$. Then $f$ is $p$-Dunford integrable.
\end{theorem}
\begin{proof} For each $n\in \Nat$, define the absolutely convex set
$$
	C_n:=\Big\{x^*\in X^*: \, \Big(\int_\Omega|\langle f,x^*\rangle|^p\, d\mu\Big)^{1/p} \leq n\Big\}.
$$
We will prove that {\em $C_n$ is $w^*$-closed}.

To this end, we first use the scalar measurability of~$f$ to find 
an increasing sequence $(E_m)$ of measurable sets with $\Omega=\bigcup_{m\in \Nat} E_m$ 
such that $f_m:=f\chi_{E_m}$ is scalarly bounded for all $m\in \Nat$
(see e.g. \cite[Proposition~3.1]{Musial}). The $\mu$-PIP of~$X$
ensures that each $f_m$ is $p$-Pettis integrable.

Fix $x^*\in \overline{C_n}^{w^*}$ and $m\in \Nat$. By the $p$-Pettis integrability of~$f_m$, the operator $S^p_{f_m}:X^*\to L^p(\mu)$ is
$w^*$-$w$-continuous (just make the obvious changes to the proof of the case $p=1$, see e.g. \cite[Theorem~4.3]{Musial}).
Hence
$$
	S^p_{f_m}(x^*)\in
	\overline{S^p_{f_m}(C_n)}^{w}= 
	\overline{S^p_{f_m}(C_n)}^{\|\cdot\|_{L^p(\mu)}}
$$
(bear in mind that $S^p_{f_m}(C_n)$ is convex). 
Therefore, there is a sequence $(x_k^*)$ in~$C_n$ such that 
$$
	\lim_{k\to \infty}\|\langle f_m,x_k^*\rangle - \langle f_m,x^*\rangle\|_{L^p(\mu)} \to 0.
$$
In particular, $\|\langle f_m,x^*\rangle\|_{L^p(\mu)}\leq n$. As $m\in \Nat$ is arbitrary, an appeal
to the monotone convergence theorem yields
$$
	\int_\Omega |\langle f,x^*\rangle|^p \, d\mu=\lim_{m\to \infty}\int_\Omega |\langle f_m,x^*\rangle|^p \, d\mu\leq n^p,
$$
so that $x^*\in C_n$. This proves that $C_n$ is $w^*$-closed.

Note that $\Gamma=\bigcup_{n\in \Nat}\Gamma\cap C_n$ and $C_n\sub C_{n+1}$ for all $n\in \Nat$. Since $\Gamma$ is $w^*$-thick,
there is $n\in \Nat$ such that $\overline{{\rm aco}}^{w^*}(\Gamma\cap C_n)$ contains a ball centered at~$0$, and so does $C_n$ (because
it is absolutely convex and $w^*$-closed). 
That is, there is $\delta>0$ such that $\delta x^* \in C_n$ for every $x^*\in B_{X^*}$.
This clearly implies that $f$ is $p$-Dunford integrable. The proof is finished. 
\end{proof}

The space $X$ is said to have property~($\mathcal{E}$) (of Efremov) if for every convex bounded set $C \sub X^*$, any element of the $w^*$-closure
of~$C$ is the $w^*$-limit of a sequence contained in~$C$. This class of Banach spaces has been studied in \cite{pli3,pli-yos-2}. It contains all
Banach spaces having $w^*$-angelic dual and, in particular, all weakly compactly generated spaces. Every Banach space
having property~($\mathcal{E})$ also satisfies the so-called Mazur's property and, therefore, has the PIP (see \cite{edgar2}).

For Banach spaces having property~($\mathcal{E}$) the scalar measurability assumption in 
Theorem~\ref{theorem:pDunford-thick} is redundant and we have the following result.

\begin{corollary}\label{corollary:pDunford-thick}
Suppose $X$ has property~($\mathcal{E}$). Let $f:\Omega \to X$
be a function for which there is a $w^*$-thick set $\Gamma\sub X^*$ such that
$Z_{f,\Gamma}\sub L^p(\mu)$. Then $f$ is $p$-Dunford integrable.
\end{corollary}
\begin{proof}
Since $\Gamma$ is $w^*$-thick, the set $\overline{{\rm aco}}^{w^*}(\Gamma)$ contains a ball centered at~$0$ and, in particular,
$\Gamma$ separates the points of~$X$. Since $\langle f,x^* \rangle$ is measurable for all $x^*\in \Gamma$ and $X$ has property~($\mathcal{E}$),
we can apply \cite[Proposition~12]{pli3} to conclude that $f$ is scalarly measurable. Theorem~\ref{theorem:pDunford-thick} now 
ensures that $f$ is $p$-Dunford integrable.
\end{proof}
 
\begin{remark}\label{remark:thickPIP}
For $1<p<\infty$, the conclusion of Theorem~\ref{theorem:pDunford-thick} and Corollary~\ref{corollary:pDunford-thick}
can be strengthened to ``$f$ is $p$-Pettis integrable'' (by Corollary~\ref{corollary:PIP}).
\end{remark} 
 
Our last example is based on a construction given in \cite[Example~8]{raj-rod} and shows that 
Corollary~\ref{corollary:pDunford-thick} does not work for arbitrary Banach spaces.
Here $\ell^1([0,1])$ is seen as the dual of~$c_0([0,1])$, so that the set $c_0([0,1]) \sub \ell^1([0,1])^*$ ($=\ell^\infty([0,1])$) 
is $w^*$-thick (see e.g. \cite[Theorem~1.5, (a)$\Leftrightarrow$(c)]{nyg}).

\begin{example}\label{example:NonSeparable}
There is a function $f:[0,1]\to \ell^1([0,1])$ such that: 
\begin{enumerate}
\item[(i)] $\langle f,\varphi\rangle=0$ a.e. for every $\varphi\in c_0([0,1])$;
\item[(ii)] $f$ is not Dunford integrable.
\end{enumerate}
\end{example}
\begin{proof} For each $t\in [0,1]$, let $e_t\in \ell^1([0,1])$ be defined by
$e_t(s):=1$ if $t=s$ and $e_t(s):=0$ if $t\neq s$. Define
$$
	f:[0,1]\to \ell^1([0,1]), \quad
	f(t):=h(t)e_{t},
$$ 
where $h:[0,1]\to \erre$ is any function such that $h\not \in L^1[0,1]$. Condition~(i) holds
because $\langle f,\varphi\rangle$ vanishes outside
of a countable subset of~$[0,1]$. On the other hand, $f$ is not Dunford integrable because
the functional $\chi_{[0,1]} \in \ell^\infty([0,1])$ satisfies $\langle f,\chi_{[0,1]}\rangle=h$.
\end{proof}


\def\cprime{$'$}\def\cdprime{$''$}
  \def\polhk#1{\setbox0=\hbox{#1}{\ooalign{\hidewidth
  \lower1.5ex\hbox{`}\hidewidth\crcr\unhbox0}}} \def\cprime{$'$}
\providecommand{\bysame}{\leavevmode\hbox to3em{\hrulefill}\thinspace}
\providecommand{\MR}{\relax\ifhmode\unskip\space\fi MR }
\providecommand{\MRhref}[2]{%
  \href{http://www.ams.org/mathscinet-getitem?mr=#1}{#2}
}
\providecommand{\href}[2]{#2}


\bibliographystyle{amsplain}

\end{document}